\documentclass[12pt]{article}

\usepackage{amsthm}
\usepackage{amsmath}
\usepackage{epsfig}
\usepackage{amssymb}
\usepackage{latexsym}
\usepackage[utf8]{inputenc}
\usepackage{url}

\newtheorem{definition}{Definition}
\newtheorem{theorem}[definition]{Theorem}

\newtheorem{lemma}[definition]{Lemma}
\newtheorem{conjecture}{Conjecture}

\newtheorem{observation}[definition]{Observation}
\let\epsilon=\varepsilon

\let\rho = \varrho

\begin{document}
\title{Avoiding $5$-circuits in a $2$-factor of cubic graphs}

\author{
Barbora Candráková${}^1$,
Robert Lukoťka${}^2$
\\[3mm]
\\{\tt candrakova@dcs.fmph.uniba.sk}
\\{\tt robert.lukotka@truni.sk}
\\[5mm]
${}^1$ Faculty of Mathematics, Physics and Informatics\\
Comenius University \\
Mlynská dolina, 842 48 Bratislava\\
${}^2$ 
Faculty of Education\\
Trnava University in Trnava\\
Priemyselná 4, 918 43 Trnava
}

\maketitle

\begin{abstract}
We show that every bridgeless cubic graph $G$ on $n$ vertices other than the Petersen graph 
has a $2$-factor with at most $2(n-2)/15$ circuits of length $5$.
An infinite family of graphs attains this bound.
We also show that $G$ has a $2$-factor with at most $n/5.8\overline{3}$ odd circuits.
This improves the previously known bound of $n/5.41$ [Lukoťka, Máčajová, Mazák, Škoviera:
Small snarks with large oddness, {\tt arXiv:1212.3641 [cs.DM]}].
\end{abstract}



\section{Introduction}

Petersen \cite{petersen} showed that every cubic graph
without a bridge has a $1$-factor. We can restate
many well known problems in graph theory 
in terms of $1$-factors in cubic graphs. It often happens that we study $1$-factors through their complementary $2$-factors as certain properties of graphs are better described this way. An example is the following equivalence:
A cubic graph is $3$-edge-colourable if and only if it has
a $2$-factor with even circuits only. 

\emph{Snarks}, connected bridgeless cubic graphs that are not $3$-edge-colourable,
are an intensively studied class of graphs. Many important problems and conjectures can be reduced to snarks:
the $4$-colour theorem, Tutte's $5$-flow conjecture, or the cycle double 
cover conjecture \cite{huck,kochol}.
Problems regarding $1$-factors (and thus $2$-factors) tend to be challenging; a conjecture that there is
exponentially many perfect matchings from 1970s
has been proven only recently \cite{expfactor}.

The minimum number of odd circuits in a
$2$-factor of a bridgeless cubic graph $G$ naturally describe how uncolourable $G$ is.
This parameter is called \emph{oddness} and is denoted by $\omega(G)$. Since every cubic
graph has an even number of vertices, its oddness must be
even. Oddness is an interesting property to consider since the  $5$-flow conjecture and the cycle double 
cover conjecture are proven for snarks of small oddness \cite{jaeger, huckkochol, hag}.
Other parameters quantifying the uncolourability 
of cubic graphs can be related to oddness. We refer the reader to the paper \cite{steffen}.

The presence
of short circuits in a $2$-factor is another interesting property to study with several applications. We can use standard reduction
theorems for circuits of length $2$, $3$, and $4$ \cite{snarks}. On the other hand, circuits of length $5$ pose main obstacles in several problems (e.g. \cite{kral, boyd}) and our knowledge on how to avoid them was very limited except of some well defined situations.

The Petersen graph has two $5$-circuits in all of its $2$-factors. Mkrtchyan and Petrosyan
conjectured that this is the only graph that has only $5$-circuits in each $2$-factor \cite{pgc}.
DeVos immediately confirmed this conjecture \cite{devos}. K\" undgen and Richter
\cite{kundgen} showed that each cubic graph has a $2$-factor with circuit of length at least $7$ except for graphs on four or six vertices, and except for the Petersen graph. This makes
the Petersen graph very special regarding $5$-circuits in its $2$-factor.

Lukoťka, Máčajová, Mazák, and Škoviera \cite{LMMS} studied how small 
a cubic graph with given oddness can be. They developed a method to avoid $5$-circuits
in $2$-factors of cubic graph and
proved that a bridgeless cubic graph $G$ 
not isomorphic to the Petersen graph has at least $5.41 \cdot \omega(G)$ vertices.
They proposed the following conjecture.
\begin{conjecture}\cite{LMMS} \label{con}
Let $G$ be a $2$-edge-connected cubic graph with oddness $\omega(G)$.
Then $G$ has at least $7.5\cdot \omega(G)-5$ vertices. 
\end{conjecture}
\noindent The bound in this conjecture is attained for $\omega(G)\equiv 2 \ (\bmod 4)$.

We refine the methods introduced in \cite{LMMS}. Our approach allows us
to deal separately with certain subgraphs isomorphic 
to three graphs created from the Petersen graph and we improve the bound to $5.8\overline{3} \cdot \omega(G)$. 

We bound the minimal number of $5$-circuits in a bridgeless cubic graph $G$.
We will call this parameter \emph{$5$-cyclicity} and denote it by $\omega_5(G)$.
We prove that the $5$-cyclicity of a $2$-edge-connected cubic graph $G$ on $n$ vertices not isomorphic to the Petersen graph
is at most $2(n-2)/15$. This improves the bound $11n/75$ that can be deduced from \cite{LMMS}.
We construct an infinite family of cubic graphs that attains this bound and
observe that these graphs satisfy Conjecture \ref{con}. This adds additional support for the conjecture.
Moreover, we prove bound $n/9$ for a wide class of $3$-edge-connected cubic graphs.
We prove the bound $n/10$ for cyclically $4$-edge-connected graphs of girth $5$ (also called ``non-trivial snarks'').

All proofs in this work are constructive. The required $2$-factors can be produced by the algorithm 
for minimum-weight perfect matching in bridgeless graphs and 
contain no triangles. The only apparent obstacle, the reductions of colourable graphs separated by small cuts can be circumnavigated by restricting us to cuts separating less than $14$ vertices.

\section{The improved bound on oddness}

First we proof an improved bound on the oddness of a cubic graph. The bound on the $5$-cyclicity will use similar 
(and slightly simpler) argument.
\begin{theorem}
A $2$-edge-connected cubic graph $G$ not isomorphic to the Petersen graph with oddness $\omega(G)$ has at least $5.8\overline{3} \cdot \omega(G)$ vertices.
\label{thm1}
\end{theorem}

We will use the following lemmas from \cite{LMMS} in the proof of the theorem.
\begin{lemma}
For every snark $G$ there exists a snark $G'$ of order not exceeding that of $G$ such that $\omega(G')\ge\omega(G)$ and the girth of $G'$ is at least $5$.
\label{lemmaGirth}
\end{lemma}
We say a graph is \emph{colourable} if there exists a $3$-edge colouring of the graph, otherwise we say it is \emph{uncolourable}.
\begin{lemma}
For every snark $G$ there exists a snark $G'$ of order not exceeding that of $G$ such that $\omega(G')\ge\omega(G)$ and every $2$-edge-cut in $G'$ separates two uncolourable subgraphs of $G'$.
\label{lemma2edgeCut}
\end{lemma}
\begin{lemma}
For every snark $G$ there exists a snark $G'$ of order not exceeding that of $G$ such that $\omega(G')\ge\omega(G)$ and every $3$-edge-cut in $G'$ separates two uncolourable subgraphs of $G'$.
\label{lemma3edgeCut}
\end{lemma}

\begin{proof}[Proof of Theorem \ref{thm1}]
Let $G$ be the smallest (in the number of vertices) $2$-edge-connected cubic graph with given oddness.
We can assume by the above lemmas that $G$ does not contain circuits shorter than $5$ and every $2$ and $3$-edge-cut in $G$ separates two uncolourable subgraphs of $G$ (if $G$ reduces to the Petersen graph, then $G$ has oddness $2$, at least $12$ vertices, and satisfies Theorem \ref{thm1}). 

Let ${\cal P}_2$ be the set of subgraphs of $G$ isomorphic to the graph $P_2$ (the Petersen graph with an edge subdivided twice).
\begin{figure}
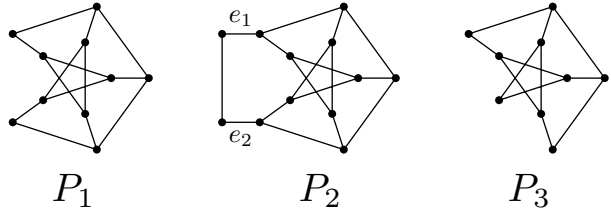

\center
\begin{tabular}{ccccc}
\includegraphics{special-1} & &\includegraphics{special-2}& &\includegraphics{special-3}\\
\end{tabular}
\caption{Graphs $P_1$, $P_2$, $P_3$. Subgraphs isomorphic to these graphs will be investigated separately.}
\label{fig}
\end{figure}
Let ${\cal P}_1$ be the set of subgraphs of $G$ isomorphic to the graph $P_1$ (the Petersen graph with one edge removed) that cannot be extended to a subgraph of $G$ isomorphic to $P_2$. Let ${\cal P}_3$ be the set of subgraphs of $G$ isomorphic to the graph $P_3$ (the Petersen graph with one vertex removed) that cannot be extended to a subgraph of $G$ isomorphic to $P_1$. (These subgraphs are defined in Figure~\ref{fig}.) Let ${\cal P}={\cal P}_1\cup{\cal P}_2\cup{\cal P}_3$. All subgraphs from ${\cal P}$ are not $3$-edge-colourable.

\begin{lemma}
Subgraphs of $G$ from ${\cal P}$ are pairwise disjoint, except when $G$ is a graph on $22$ vertices that
contains two subgraphs isomorphic to $P_2$ intersecting in two vertices and one edge.
This graph fulfils Theorem \ref{thm1}.
\label{ldis}
\end{lemma}
\begin{proof}
Let $E_2$ be the set of edges of $G$ that are in some $2$-edge-cut of $G$. Since all vertices of $P_1$, $P_2$, and $P_3$ are of degree at least two, two intersecting subgraphs $S_1,S_2 \in {\cal P}$ must share a common edge. We take an edge $e=uv \in E(S_1) \cap E(S_2)$. Suppose first that $deg(u)=deg(v)=2$ in both $S_1$ and $S_2$. Then $G$ is the exceptional graph from the lemma statement. We can easily check that Theorem \ref{thm1} holds for this graph.

Suppose that $deg(u)=2$ and $deg(v)=3$ in $S_1$. Since the $deg(v)$ in $S_2$ is at least $2$, there must exist an edge $e'=vw$ that also belongs to $E(S_1) \cap E(S_2)$. We get that $deg(v)=3$ and $deg(w)=3$ in $S_1$ from the structure of the subgraphs from ${\cal P}$. We choose the edge $e'$ as our new edge $e$.

Apparently, the edge $e$ does not belong to $E_2$. Let $S$ be the largest connected subgraph (with respect to the number of edges) of $G$ that contains $e$ and contains no edge from $E_2$. Subgraph $S$ is unique. The properties of graphs $P_1$, $P_2$, and $P_3$ guarantee that $|V(S)|\ge 10$. 

Suppose first that $|V(S)|=10$. Than $S_1$ must be an induced subgraph of $S_2$ or vice versa. This contradicts the definition of subgraphs from ${\cal P}$. On the other hand, suppose that $|V(S)|>10$. Then $S_1$ and $S_2$ are isomorphic to $P_3$. Let $E_3$ be the set of edges that are contained in some minimal edge-cut with independent edges (with respect to set inclusion) of size at most $3$. Let $S'$ be the largest connected subgraph (with respect to the number of edges) of $G$ that contains $e$ and contains no edge from $E_3$. Due to properties of $P_3$ and due to the fact that $S_1$ and $S_2$ cannot be extended to a subgraph isomorphic to $P_1$ we know that $|V(S')|=9$ and $S=S_1=S_2$.
\end{proof}

Hence we can assume all subgraphs from ${\cal P}$ are disjoint. Let $M$ be a perfect matching of $G$ and $F_M$ the complementary $2$-factor. Let ${\cal C}(M)$ be the set of circuits of $F_M$. The symbol $|C|_o$ will denote the length of the circuit $C\in {\cal C}$ when $C$ is odd, otherwise we define $|C|_o=|C|+7$.
Let
$$
I(M)=\sum_{C \in {\cal C}(M)} \frac{7-|C|_o}{2}.
$$
We can bound the number of odd circuits of ${\cal C}(M)$ by the value $I(M)$ where $M$ is chosen in such a way that $I(M)$ is small enough.

Let $C$ be a circuit of $G$ and let $S\in {\cal P}$. We say that $C$ \emph{goes through} $S$ if $C$ intersects $S$ in at least two edges. We split the function $I$ into several parts according to this property.

For $C\in {\cal C}(M)$ and $S$ such that $C$ goes through $S$ we define $I(C,S,M)=(7-|C|_o)/(2k)$ where $k$ is the total number of subgraphs from ${\cal P}$ that $C$ goes through, otherwise we define $I(C,S,M)=0$.
Let $I(S,M)=\sum_{C\in {\cal C}(M)} I(C,S,M)$.

For $C\in {\cal C}(M)$ that does not go through any subgraph from ${\cal P}$ we define $I(C,M)=(7-|C|_o)/2$, otherwise we set $I(C,M)=0$.

We can rewrite the invariant $I$ as 
\begin{eqnarray}
I(M)=\sum_{C \in {\cal C}(M)} I(C,M) + 
\sum_{S\in{\cal P}_1} I(S, M)+\sum_{S\in{\cal P}_2} I(S, M)+\sum_{S\in{\cal P}_3} I(S, M). \label{eqn}
\end{eqnarray}
The splitting allows us to bound the value of $I(M)$ step by step.

We define a linear function and we minimize it over the perfect matching polytope \cite{edmonds} of $G$, denoted by ${\cal M}(G)$. For a point $p\in{\cal M}(G)$ let $p_e$ be the weight corresponding to the edge $e$. For each $S\in{\cal P}_1 \cup {\cal P}_2$ let $e_S$ be one arbitrary edge on the boundary of $S$ (an edge with exactly one vertex belonging to $S$). For each $S\in{\cal P}_3$ we define $E_S$ either to be $\emptyset$, or to contain two edges from the boundary of $S$ whenever possible such that
\begin{enumerate}
\item these two edges are not in the same $7$-circuit that goes through $S$
and can be contained in a $2$-factor of $G$  and
\item these two edges are not in a $9$-circuit that goes through $S$ and some other subgraph from ${\cal P}$.
\end{enumerate}
The set of subgraphs from ${\cal P}_3$ with $E_S\neq \emptyset$ and with $E_S=\emptyset$ will be denoted by ${\cal P}_{3a}$ and ${\cal P}_{3b}$ respectively. 
Let ${\cal C}_5$ be the set of circuits of length $5$ in $G$ not going through any subgraph from ${\cal P}$. For each $C \in {\cal C}_5$ let $E_C$ be the set of edges on the boundary of $C$.

We define the linear function as follows:
$$
f(p)=\left(\sum_{C \in {\cal C}_5} \frac{1}{4}\sum_{e\in E_C} p_e  \right) + \sum_{S\in {\cal P}_1} 2p_{e_S} +
\sum_{S\in {\cal P}_2} p_{e_S} + \left( \sum_{S\in {\cal P}_{3a}} \sum_{e\in E_S} p_e\right).
$$
Since we maximize over a polytope, the optimal value is attained at some vertex of the polytope - some perfect matching in case of the perfect matching polytope. We find a point $p$ in ${\cal M}(G)$ such that $f(p)$ is minimal. 
Due to the characterization of the perfect matching polytope, 
the point $(1/3, 1/3, \dots ,1/3)$ lays within the polytope \cite{edmonds}. 
Hence, for optimal point $p$ of ${\cal M}(G)$
\begin{eqnarray}
f(p) \le 5/12 |{\cal C}_5|  + 2/3 |{\cal P}_1|+ 1/3 |{\cal P}_2| + 2/3|{\cal P}_{3a}|.
\label{eqn2}
\end{eqnarray}

Let $M$ be a perfect matching satisfying $f(M)=f(p)$ such that the number of pairs $(S,C)$, 
where $S\in {\cal P}_2$ and $C\in {\cal C}(M)$ goes through $S$, is minimal. 
We show that $I(M) \le f(M) - |{\cal C}_5|/4 +|{\cal P}_{3b}|$. To do this, we use the equation (\ref{eqn}) and bound the summands of $I(M)$ one by one.

\medskip
\noindent {\bf Part 1:} First, we consider the value $I(C,M)$ for $C\in{\cal C}(M)$. 
For a circuit $C\in {\cal C}(M)$ of length more than $5$ or a circuit going through some subgraph from ${\cal P}$
we have $I(C,M)\le 0$ by definition. The only remaining case to consider is a circuit not going through any subgraph from ${\cal P}$ and of length $5$ (that is $C\in {\cal C}_5$). 
By definition we have $I(C,M)=1$.
All boundary edges of $C$ belong to $M$, that is $\sum_{e\in E_C} p_e=5$.
Because $\sum_{e\in E_C} p_e\ge 1$ 
we know that $I(C,M)\le -1/4+1/4\sum_{e\in E_C} p_e$.
Altogether 
$$
\sum_{C\in{\cal C}(M)} I(C, M) \le 
\left(\sum_{C \in {\cal C}_5} \frac{1}{4}\sum_{e\in E_C} p_e  \right)-\frac{1}{4}|{\cal C}_5|.
$$

\medskip
\noindent {\bf Part 2:} Consider a subgraph $S \in {\cal P}_1$. Suppose that $p_{e_S}=1$. This means that both edges on the boundary of $S$ belong to $M$. Therefore exactly two $5$-circuits go through $S$. By definition we have $I(S,M)=2$. 

On the other hand, if $p_{e_S}=0$ (none of the boundary edges belongs to $M$), then two circuits of ${\cal C}(M)$, $C_1$ and $C_2$, go through $S$ such that $|C_1|=5$ and $|C_2|\geq 8$ (If $C_2$ had length $7$, then $S$ could be extended to a subgraph isomorphic to $P_2$). Clearly, $I(C_1,S,M)=1$. We bound the value of $I(C_2,S,M)$.

\begin{observation}
For circuit $C_2$ we have $I(C_2,S,M) \le -1$.
\label{oc2}
\end{observation}
\begin{proof}
If $|C_2|=9$, then $C_2$ cannot go through any other subgraph of ${\cal P}$ other than $S$. Indeed, this could only happen if $S$ was connected to $S_2 \in {\cal P}_3$ by two edges which would imply a bridge. This shows that $I(C_2,S,M)=-1$.

Let $k$ be the number of subgraphs from ${\cal P}$ that $C_2$ goes through. Suppose that $|C_2|>9$ and $C_2$ is odd. Then $C_2$ contains at least $5$ vertices form $S$ and at least $4$ vertices from each other subgraph that $C_2$ goes through. Therefore $|C_2|\ge4k+1$. Since $k$ is an integer, $k\le \lfloor(|C_2|-1)/4\rfloor$ and we have
$$
I(C_2,S,M)=(7-|C_2|)/2k\le (7-|C_2|)/ 2\lfloor(|C_2|-1)/4\rfloor \le-1.
$$

Suppose that $C_2$ is even. We have, again, 
$$
I(C_2,S,M)=(7-(|C_2|+7)/2k \le (-|C_2|)/ 2\lfloor(|C_2|-1)/4\rfloor \le-1.\qedhere
$$
\end{proof}
\noindent This shows that $I(S,M)\le 0$ when $p_{e_S}=0$. 

Merging the situations $p_{e_S}=0$ and $p_{e_S}=1$ into a single inequality yields $I(S,M) \le 2p_{e_S}$.
Altogether 
$$
\sum_{S\in{\cal P}_1} I(S,M) \le \sum_{S\in{\cal P}_1} 2p_{e_S}.
$$

\medskip
\noindent {\bf Part 3:} Suppose we have a subgraph $S \in {\cal P}_2$. Similarly as before, if $p_{e_S}=1$, then $I(S,M)=1$ as we have a $5$-circuit and a $7$-circuit going through $S$. 

If $p_{e_S}=0$, then due to the minimality of the number of circuits going through $S$, 
only two circuits go through $S$: one circuit of length $5$, and the second of length at least $9$. By similar argumentation as in Observation \ref{oc2} we have
$I(S,M)\le 0$. 
In both cases we get $I(S,M) \le p_{e_S}$ and altogether
$$
\sum_{S\in{\cal P}_2} I(S,M) \le  \sum_{S\in{\cal P}_2} p_{e_S}.
$$

\medskip
\noindent {\bf Part 4a:} Suppose we have a subgraph $S \in {\cal P}_{3a}$.   
We have two edges in $E_S$ and the value of $\sum_{e\in E_S} p_e$ can be either $0,1,$ or $2$. 

If $\sum_{e\in E_S} p_e=0$, then there can be at most one $5$-circuit $C_1$ and one other circuit $C_2$ going through $S$. Clearly, $I(C_1,S,M)= 1$. We know that $C_2$ is not of length $7$, and if it is of length $9$,
then it only goes through one subgraph from ${\cal P}$ (conditions on the set $E_S$). 
We obtain $I(C_2,S,M)\le -1$ by similar arguments as in Observation \ref{oc2}.  

If $\sum_{e\in E_S} p_e=1$, then there can be at most one $5$-circuit $C_1$ and one other circuit $C_2$ of length more than $5$ going through $S$. Clearly, $I(C_1,S,M)= 1$ and $I(C_2,S,M)\le 0$.  

If $\sum_{e\in E_S} p_e=2$, then there is only one $9$-circuit $C_1$ going through $S$.
 Clearly, $I(C_1,S,M)=-1$.  

Putting all the cases together, we have $I(S,M) \leq \sum_{e\in E_S} p_e$. Altogether
$$
\sum_{S\in{\cal P}_{3a}} I(S,M) \le \sum_{S\in{\cal P}_{3a}} p_{e_S}.
$$

\noindent {\bf Part 4b:} Suppose we have a subgraph $S \in {\cal P}_{3b}$ and each pair of boundary edges
are either in a common $2$-factor $7$-circuit
or in a common $9$-circuit going through two subgraphs from ${\cal P}$.

First let us deal with the latter case. We show that no two boundary edges of $S\in{\cal P}_{3b}$ lay in a common
$9$-circuit going through another subgraph $S_2$ from ${\cal P}$. Suppose otherwise. The subgraph $S_2$ can be isomorphic only to $P_3$ as an isomorphism with $P_1$ would imply a bridge in the graph.
Let $a_1,a_2,a_3$ be the edges on the boundary of $S$ and $b_1,b_2,b_3$ the edges on the boundary of $S_2$. Since there is a $9$-circuit going through both $S$ and $S_2$ we can assume that $a_1=b_1$ and the edges $a_2$ and $b_2$ are incident to a common vertex $v$. If $a_3$ and $b_3$ do not share a common vertex, we can choose $a_1$ and $a_3$ into $E_S$, which is a contradiction with $S$ belonging to ${\cal P}_{3b}$. Therefore, there must be a vertex $v_2$ incident to both $a_3$ and $b_3$. The vertices $v$ and $v_2$ cannot have a common neighbour (a bridge would be created) but then we can choose the edges $a_2$ and $a_3$ into $E_S$ as they are not contained in a $7$-circuit and since vertices $v$ and $v_2$ cannot be in a subgraph from ${\cal P}$, the edges $a_2$ and $a_3$ are not in a common $9$-circuit and thus we can choose them into $E_S$.

We are left with the case where each pair of boundary edges is in a $7$-circuit of some $2$-factor of $G$. Let $v_1,v_2$, and $v_3$ be the vertices if degree two in $S$ and $w_1,w_2$, and $w_3$ their neighbours outside $S$.
The fact that each pair of boundary edges must lay in a $7$-circuit implies only 
two possible configurations: $P_{3b,1}$ -- all the vertices $w_1,w_2$, and $w_3$ are adjacent to a common vertex; $P_{3b,2}$ -- each pair of the vertices $w_1,w_2$, and $w_3$ is adjacent to different vertex $x_1, x_2,$ or $x_3$ (see Figure~\ref{figspec}).
\begin{figure}
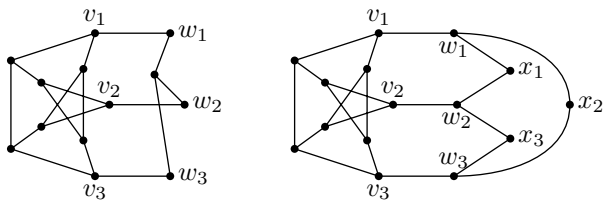

\label{figspec}
\center
\begin{tabular}{ccc}
\includegraphics{p3cases-1} & &\includegraphics{p3cases-2}\\
\end{tabular}
\caption{Configurations $P_{3b,1}$ and $P_{3b,2}$.}
\end{figure}
Only the configuration $P_{3b,2}$ can contain boundary edges of $S$ in a common $7$-circuit of some $2$-factor. Regardless of the $2$-factor, at most two circuits may go through $S$
and at most one of them is of length $5$. Therefore $I(S,M) \leq 1$ and altogether
$$
\sum_{S\in{\cal P}_{3b}} I(S,M) \le |{\cal P}_{3b}|.
$$

We make an important observation about the subgraph $S\in {\cal P}_{3b}$ in the $P_{3b,2}$ configuration that we will use later.
\begin{observation} Vertices $w_1$, $w_2$, and $w_3$ (neighbours of $S$) satisfy the following properties:
They can be in at most one $5$-circuit.
They are not inner vertices of any subgraph from ${\cal P}$.
None of these vertices has a neighbour from another subgraph ${\cal P}_{3b}$.
\label{obs}
\end{observation}
\begin{proof}
Any $5$-circuit containing $w_1$ must also contain $x_1$ and $x_2$. 
There is at most one possibility how to extend this path into $5$-circuit.
Therefore $w_1$ (and, by symmetry, $w_2$, and $w_3$) is only in one $5$-circuit.

Suppose for contradiction that $w_1$ is an inner vertex of a subgraph $S_2$ from ${\cal P}$.
Vertex $w_1$ has degree at most $2$ in $S_2$. Therefore $S \not \in {\cal P}_2$ because $P_2$ has no inner vertex of degree $2$.
Subgraphs from ${\cal P}$ are disjoint and have minimal degree $2$. This means that
$x_1,x_2$ are two inner vertices from $S_2$. Since any neighbour of a vertex of degree $2$ in $P_1$
or $P_2$ has degree $3$, we know that $w_2 \in S_2$. The vertex $w_2$ can have degree at most $2$ in $S_2$,
which is a contradiction since neither $P_1$ nor $P_2$ contain two vertices of degree $2$ in distance $2$.
By symmetry, none of the vertices $w_1$, $w_2$, and $w_3$ is an inner vertex of a subgraph from ${\cal P}$.
 
Suppose for contradiction that $w_1$ has a neighbour that is in subgraph $S_3$ from ${\cal P}_{3b}$.
We can assume that $x_2$ is the neighbour. Since $x_2$  must have degree $2$
in $S_3$ and all vertices in distance at most $2$ in $S_3$ have degree $3$ in $S_3$ 
($S_3$ is isomorphic to $P_3$) we know that $w_3\in S_3$, and consequently $v_3\in S_3$.
This is a contradiction with the fact that subgraphs from ${\cal P}$ are disjoint.
By symmetry, $w_1$, $w_2$, and $w_3$ ale all neighbours only to one subgraph from ${\cal P}_{3b}$,
that is $S$.
\end{proof}

Putting all derived equations for the invariant $I$ together with the equation (\ref{eqn}) we get that
\begin{eqnarray}
I(M) \le f(M) - |{\cal C}_5|/4 +|{\cal P}_{3b}|. \label{eqnX}
\end{eqnarray}
By (\ref{eqn2}) and (\ref{eqnX}) we get that
\begin{eqnarray}
I(M) \le \frac{1}{6} |{\cal C}_5| +  2/3 |{\cal P}_1|+ 1/3 |{\cal P}_2| + 2/3|{\cal P}_{3a}|+|{\cal P}_{3b}|. \label{eqn3}
\end{eqnarray}
Let $M(G)$ denote the perfect matching $M$ constructed in this way and $I(G)$ denote the value of $I(M(G))$.
We can proceed to the final step. We bound $V(G)/\omega(G)$ in terms of $I(G)/V(G)$.

Let $k$ be the number of  odd circuits of $M(G)$.
Let $m$ be the number of vertices in even circuits of $M(G)$. 
For an even circuit $(7-|C|_o)/2=-|C|/2$.
Therefore  $I(G)=-m/2+7k/2-(|V(G)|-m)/2=7k/2-|V(G)|/2$. (Recall the very first definition of $I$.)
Since $\omega(G)\le k$ we have $\omega(G) \le \frac{2I(G)+|V(G)|}{7}$. Now we are ready to bound $V(G)/\omega(G)$.
\begin{eqnarray}
\frac{V(G)}{\omega(G)} \ge \frac{|V(G)|}{\frac{2I(G)+|V(G)|}{7}}=\frac{7}{1+2\frac{I(G)}{|V(G)|}}. 
\label{eqn41}
\end{eqnarray}

To bound $I(G)/|V(G)|$ we must first bound $|V(G)|$ in terms of $|{\cal C}_5|$, $|{\cal P}_1|$, $|{\cal P}_2|$, $|{\cal P}_{3a}|$, and $|{\cal P}_{3b}|$.
\begin{lemma}\label{l3}
$V(G) \ge 5/3 |{\cal C}_5| + 10 |{\cal P}_1| + 10 |{\cal P}_2| + 9 |{\cal P}_{3a}|+10|{\cal P}_{3b}|$.
\end{lemma}
\begin{proof}
Recall that the graph $G$ is the smallest graph with chosen oddness. By reduction lemmas, it does not contain circuits shorter than $5$, and every $2$-edge-cut and $3$-edge-cut in $G$ separates 
two uncolourable subgraphs of $G$.

We will count separately the vertices from ${\cal P}_1$, the inner vertices from ${\cal P}_2$, and the vertices from ${\cal P}_3$. This is $10 |{\cal P}_1| + 10 |{\cal P}_2| + 9 |{\cal P}_{3a}|+9|{\cal P}_{3b}|$ vertices. Other vertices will be called \emph{uncounted}.

Consider an uncounted vertex $v$. We know that the girth of $G$ is at least $5$, therefore the neighbourhood of $v$ of size $2$ is perfectly determined: $v$ is neighbour to three vertices $v_1$, $v_2$, and $v_3$. Vertex $v_i$ for $i\in\{1,2,3\}$ has two neighbours other than $v$, we denote them $v_{i1}$ and $v_{i2}$. All these vertices must be pairwise distinct, otherwise $G$ would contain a short circuit.

If $v$ is in $6$ circuits of length $5$, then $G$ is the Petersen graph. If $v$ was in $5$ circuits of length $5$, then $\{v\} \cup \left\{v_i, v_{i1},v_{i2} \ | \ i\in \{1,2,3\}\right\}$ can by separated by a $2$-edge-cut. Since $v$ is outside ${\cal P}$, the separated graph is $3$-edge-colourable. This is a contradiction. Therefore, $v$ is in at most $4$ circuits of length $5$. 

Let $V_k$ be the set of uncounted vertices that are in exactly $k$ circuits of length $5$ for all $2\le k\le 4$, and let $V_1$ be the set of uncounted vertices that are in at most $1$ circuit of length $5$. We want to show that $|V_4|\le |V_2 \cup V_1|$. We assign each vertex $v$ from $V_4$ two distinct tuples $(v_2,C)$ where $v_2 \in V_2 \cup V_1$, $C \in C_5$, and both $v$ and $v_2$ share the circuit $C$. If the set of all assigned tuples to the vertices from $V_4$ is pairwise disjoint, we call such assignment a \emph{proper assignment}.

\begin{lemma} 
$|V_4| \le |V_2 \cup V_1|$.
\label{cl1}
\end{lemma}
\begin{proof}
We construct a proper assignment on $G$.
For each vertex $v \in V_4$ we find two vertices contained in at most two $5$-circuits one of which is common with $v$. Note that both such vertices must be uncounted (hence belong to $V_2 \cup V_1$), since each counted vertex is in at least two $5$-circuits containing only counted vertices (Figure \ref{fig}, recall that the vertices of degree $2$ in $P_2$ are uncounted).

Suppose that the vertex $v$ is in $4$ circuits of length $5$. 
We proceed to analyse the close neighbourhood of the vertex $v$.
 There are four edges induced by vertices $\left\{v_{i1},v_{i2} \ | \ i\in \{1,2,3\}\right\}$. We denote this set by $E_v$. There are four more edges adjacent to $v_{ij}$ leading outside the subgraph. We denote this set by $\overline{E}_v$. These edges must be incident to $4$ disjoint vertices $x_1,x_2,x_3,$ and $x_4$. Otherwise we can separate $v$ by an $3$-edge-cut where the separated subgraph is $3$-edge-colourable (or isomorphic to $P_3$), which contradicts our assumptions.

Let $a_v^i$ be the number of pairs $(v_{ik},e)$, where $k\in\{1,2\}$ and $e\in E_v$ is incident to $v_{ik}$. Note that $2\le a_v^i \le 4$ for each $i$ due to the girth condition on $G$. There are only two possible multisets $X_v=\{a_v^1,a_v^2,a_v^3\}$. It is either $\{4,2,2\}$ or $\{3,3,2\}$.
\begin{figure}
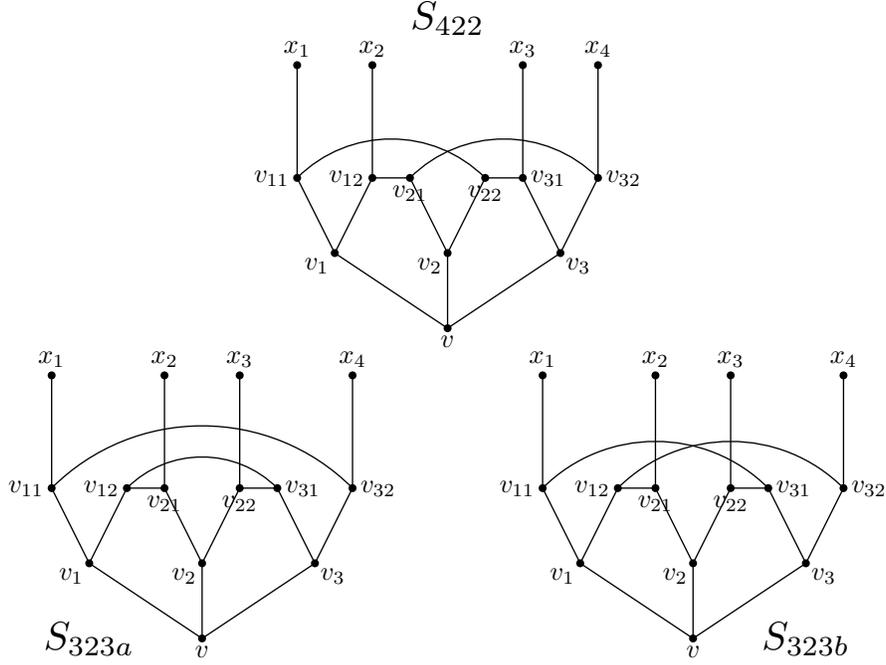

\center
\includegraphics{v45-1} \\
\includegraphics{v45-2} \ \ \ \ \ \ \ \ \ \includegraphics{v45-3}\\
\caption{Possible surroundings of an uncounted vertex contained in $4$ circuits of length $5$.}
\label{fig2}
\end{figure} 

Consider the multiset $\{4,2,2\}$, which determines a surrounding of the vertex $v$. We denote it by $S_{242}$ (Figure \ref{fig2}), as we may without loss of generality set $a_v^2=4$ and $v_{21}v_{12}, v_{22}v_{31} \in E_v$. The girth condition guarantees that the remaining edges from $E_v$ are $v_{22}v_{11}$ and $v_{21}v_{32}$. We can easily observe that for each path of length $2$ centred around $v_1$, that is $vv_1v_{11}$, $vv_1v_{12}$, and $v_{11}v_1v_{12}$, there is at most one $5$-circuit containing that path. Therefore, there are at most $3$ circuits of length $5$ going through $v_1$. 

Let us consider vertices $v_{11}$ and $v_{12}$. By analysing the neighbourhood of these vertices and using the argument that a path of length $3$ determines at most one $5$-circuit (the girth condition), we get that $v_{11}$ can be contained in at most $3$ circuits of length $5$: $v_1v_{11}v_{22}v_2v, v_1v_{11}x_1x_2v_{12}$ and $x_1v_{11}v_{22}v_{31}x_3$. However, if this vertex is contained in all of these circuit at the same time, then there exists a $3$-edge-cut separating a $3$-edge-colourable subgraph (the edge-cut is $x_2y_2$, $x_3y_3$, and $v_{32}x_4$, where $y_2$ and $y_3$ are neighbours of $x_2$ and $x_3$ that are not yet denoted, respectively). 
This is again a contradiction with Lemma \ref{lemma3edgeCut}. Similar argument can be used for 
vertices $v_{12}$, $v_{31}$, and $v_{32}$. Hence, all these vertices are contained in at most $2$ circuits of length $5$. 
Similarly, it can be shown that the vertices $v_{21}$ and $v_{22}$ are contained in at most $3$ circuits of length $5$, and the vertex $v_2$ in $4$ circuits of length $5$. 

There are four tuples that can be assigned to $v$: 
$$(v_{11}, vv_1v_{11}v_{22}v_2),  (v_{12},vv_1v_{12}v_{21}v_2),  (v_{31},vv_2v_{22}v_{31}v_3),  (v_{32},vv_2v_{21}v_{32}v_3).$$
If the vertex $v_2$ is already assigned some tuples, we choose two tuples from the above four that are not used. Otherwise we can immediately assign two tuples for both vertices $v$ and $v_2$. As the circuits of the tuples contain no other vertices from $V_4$, the set of assigned tuples always remains pairwise disjoint.

On the other hand if, $X_v=\{3,3,2\}$, then we may suppose that $a_v^2=2$. Consider the two edges from $E_v$ that are incident to the vertices $v_{21}$ and $v_{22}$. Suppose that they are both adjacent to the vertex $v_{21}$ (or $v_{22}$). Then $v$ can be separated by a $3$-edge-cut. Since the separated subgraph must not be $3$-edge-colourable, it must be isomorphic to $P_3$. However, this contradicts the assumption of $v$ being uncounted. 

Hence, we may suppose that $v_{21}v_{12}, v_{22}v_{31} \in E_v$. We get two possible surroundings: $S_{323a}$ where the two remaining edges from $E_v$ are $v_{11}v_{32}$ and $v_{12}v_{31}$, and $S_{323b}$ with $v_{11}v_{31},v_{12}v_{32} \in E_v$ as shown in Figure \ref{fig2}.

It can directly be checked in $S_{323a}$ that both vertices $v_1$ and $v_3$ are contained in at most $3$ circuits of length $5$ and $v_{11}$, and $v_{32}$ are contained in at most $2$ circuits of length $5$. We assign tuples $(v_{11},vv_1v_{11}v_{32}v_3)$ and $(v_{32},vv_1v_{11}v_{32}v_3)$ to $v$.

It can also be checked in $S_{323b}$ that all vertices $v_1$, $v_3, v_{12}$, and $v_{31}$ are contained in at most $3$ circuits of length $5$ and $v_{11}$, and $v_{32}$ are contained in at most $2$ circuits of length $5$. We assign tuples $(v_{11},vv_1v_{11}v_{31}v_3)$ and $(v_{32},vv_1v_{12}v_{32}v_3)$ to the vertex $v$.
In both cases the circuits in assigned tuples contain only one vertex from $V_4$, therefore no tuple is assigned more than once.

Let $k$ be the number of pairs $(w,C)$ where $w \in V_2 \cup V_1$ and $C$ is a $5$-circuit. The constructed assignment implies that $2|V_4| \le k$. Clearly $k \le 2|V_2 \cup V_1|$.
Hence the observation follows.
\end{proof}

Let $n$ be the number of uncounted vertices. We want to determine the number of pairs $(v,C)$ where $C$ is a $5$-circuit and $v$ is an uncounted vertex of $C$. The number of pairs is clearly at most $4|V_4|+3|V_3|+2|V_2|+|V_1|$.

According to Observation \ref{cl1} $|V_4|\leq |V_2|+|V_1|$ and
$4|V_4|+3|V_3|+2|V_2|+|V_1|\leq 3(|V_4|+|V_3|+|V_2|+|V_1|) - |V_1| = 3n-|V_1|$.

Observation \ref{obs} states that there are at least $3 |{\cal P}_{3b}|$ uncounted vertices that are in at most one $5$-circuit. Therefore
$3n-|V_1|\le 3n-3|{\cal P}_{3b}|$.

On the other hand, we have exactly $5 \cdot |{\cal C}_5|$ such pairs. Therefore, $3 n - 3|{\cal P}_{3b}| \ge 5|{\cal C}_5|$. Altogether, we get the result of the lemma $V(G) \ge 5/3 |{\cal C}_5| + 10 |{\cal P}_1| + 10 |{\cal P}_2| + 9 |{\cal P}_3a| + 
10|{\cal P}_{3b}|$.
\end{proof}

We can now conclude the proof of the Theorem \ref{thm1}. By (\ref{eqn3}) and by Lemma~\ref{l3} we get that
$$
\frac{I(M)}{V(G)} \le \frac{1/6 |{\cal C}_5| +  2/3 |{\cal P}_1|+ 1/3 |{\cal P}_2|+ 2/3|{\cal P}_{3a}|+|{\cal P}_{3b}|}
{ 5/3|{\cal C}_5| + 10 |{\cal P}_1| + 10 |{\cal P}_2|+9 |{\cal P}_{3a}|+10|{\cal P}_{3a}|} \le 1/10.
$$

Therefore by (\ref{eqn41}) we have $V(G)/\omega(G) \ge 35/6>5.8\overline{3}$.

\end{proof}

\section{Reduction lemmas for $5$-cyclicity}

We start with the analogues of Lemmas \ref{lemmaGirth}, \ref{lemma2edgeCut}, and 
\ref{lemma3edgeCut}.
\begin{lemma}
For every snark $G$ there exists a snark $G^*$ of order not exceeding that of $G$ such that 
\begin{enumerate}
\item
the girth of $G^*$ is at least $5$;
\item
$\omega_5(G^*)\ge\omega_5(G)$;
\item
if there is a $2$-factor of $G^*$ with $m$ circuits of length $5$, then there is a $2$-factor of $G$ with $m$ circuits of length $5$ and no circuits of length $3$.
\end{enumerate}
\label{lemmaGirth2}
\end{lemma}

\begin{proof}
We prove this by induction on the number of vertices.
If $G$ has no circuit of length less than $5$, then we set $G^*=G$ and the lemma holds.

Suppose there is a $2$-cycle $uv$ in $G$. Let $u',v'$ be the other neighbours of $u$ and $v$. 
We create a graph $G'$ by removing the vertices $u$ and $v$ with adjacent edges and adding the edge $u'v'$. 
By induction hypothesis there is a graph $G'^*$ without circuits of length $2$, $3$
and $4$ such that $\omega_5(G'^*)\ge\omega_5(G')$ and for a $2$-factor of $G'^*$ with $k$ circuits of length $5$ there is a $2$-factor $F$ of $G'$ with at most $k$ circuits of length $5$ and no
triangles.
We set $G^*=G'^*$.
Because $F$ does not contain $3$-circuits we can extend every $F$ 
to a $2$-factor in $G$ without adding any $5$ or $3$-circuits:
we either add $2$-circuit $uv$ into $F$ or extend the circuit 
that contained the edge $u'v'$ as $u'uvv'$. 

Each further situation uses the same induction argument. We omit it and describe only the process of creating 
a $2$-factor of $G$ from the $2$-factor of $G'$ without introducing new $3$ or $5$-circuits. 

Suppose there is a triangle $uvw$ in $G$.  
We create $G'$ by contracting $u$, $v$, and $w$ into one vertex.
By induction hypothesis $G'$ has a $2$-factor $F$ satisfying the lemma with no $3$-circuits.
We can extend $F$ to a $2$-factor of $G$ without creating circuits of length $3$ or $5$.

Finally, suppose $G$ has no circuits of lengths $2$ and $3$ but a $4$-circuit $C=v_1v_2v_3v_4$. Let
$w_1$, $w_2$, $w_3$, and $w_4$ be the neighbours of $v_1$, $v_2$, $v_3$, and $v_4$ outside of $C$, respectively
(such neighbours exist due to the fact that $G$ has no circuits shorter than $4$). Only $w_1$ and $w_3$, and $w_2$ and $w_4$ can be identical. Otherwise there would be a shorter circuit.
If both $w_1=w_3$ and $w_2=w_4$, then we produce $G'$ by deleting all denoted vertices and by adding 
a new edge between the two resulting divalent vertices. It is easy to extend the $2$-factor of $G'$ to a
$2$-factor of $G$ without adding circuits of length $3$ or $5$.

Suppose that only one pair of vertices is identical, say $w_1=w_3$. We contract $v_1$, $v_2$, $v_3$, $v_4$,
$w_1$, and $w_3$ into a new vertex to produce $G'$. Again, it is easy to extend the $2$-factor of $G'$
to a $2$-factor of $G$.

We may now expect vertices $w_1$, $w_2$, $w_3$, and $w_4$ to be pairwise distinct.
We delete $v_1$, $v_2$, $v_3$, and $v_4$ and either add edges $w_1w_2$ and $w_3w_4$, or $w_1w_4$ and $w_2w_3$
to produce $G'$.
One of these two choices always guarantees that $G'$ is bridgeless \cite{snarks}.
By induction hypothesis there is a $2$-factor $F$ of $G'$ with no $3$-circuits and
containing the same number of $5$-circuits as some $2$-factor of $G^*$.
Since $F$ has no $3$-circuits we can extend it without introducing new $3$ or $5$-circuits.
\end{proof}

\begin{lemma}
For every snark $G$ there exists a snark $G^*$ of order not exceeding that of $G$ such that $\omega_5(G^*)\ge\omega_5(G)$ and every $2$-edge-cut in $G^*$ separates two uncolourable subgraphs of $G^*$.
Moreover, if $G^*$ has a $2$-factor with $k$ circuits of length $3$ and $m$ circuits of length $5$,
then also $G$ has a $2$-factor with at most $k$ circuits of length $3$ and at most $m$ circuits of length $5$.
\label{lemma2edgeCut2}
\end{lemma}
\begin{proof}
Suppose there exists a $2$-edge-cut separating a colourable subgraph.
We choose the cut that separates the smallest (in number of vertices)
colourable subgraph. Let $v_1v_2$ and $w_1w_2$ be the cut-edges such that
$v_1$ and $w_1$ are in the colourable component of 
$G-\{v_1v_2,w_1w_2\}$. Clearly $v_1$, $v_2$, $w_1$, and $w_2$ are all distinct 
otherwise $G$ has a bridge.
Moreover $v_1$ and $w_1$ are not neighbours as this would contradict our choice of the cut.

We delete the edges $v_1v_2$ and $w_1w_2$ and add two edges $v_1w_1$ and $v_2w_2$.
This creates two components $G_1$ and $G_2$. Let $G_1$ be the component containing $v_1$ (the colourable one).
We fix a colouring $c$ of $G_1$ and we pick a $2$-factor $F'$ of $G_2$ containing the least number of $5$-circuits.

If the edge $v_2w_2$ is not part of the $2$-factor, then we put the edges from $F'$,
and the edges not coloured with the colour of $v_1w_1$ in $G_1$ into $2$-factor of $G$.
Clearly no extra $3$ or $5$-circuits were added. 
Therefore we can set $G^*=G_2$.

If the edge $v_2w_2$ belongs to the $2$-factor, then we put the following edges to the $2$-factor of $G$:
edges from $F'$, 
edges coloured with the same colour as $v_1w_1$ in $G_1$, 
edges coloured with one other fixed colour in $G_1$, and
the edges $v_1v_2$ and $w_1w_2$. No new $5$-circuits besides the $5$-circuits of $F'$ are introduced: The circuits inside $G_1$ are even. Since $v_1$ and $v_2$ are not neighbours, the only possible new $5$-circuit is
$v_1v_2w_2w_1v$ where $v$ is a common neighbour of $v_1$ and $w_1$. This is also not possible because the
edges $v_1v$ and $w_1v$ are both coloured differently from the edge $v_1w_1$. We can again set $G^*=G_2$.

The same construction can be usod to prove that to a given $2$-factor of $G^*$ we can create a $2$-factor of $G$ without introducing $3$ or $5$-circuits.
\end{proof}
\begin{lemma}
For every snark $G$ there exists a snark $G^*$ of order not exceeding that of $G$ such that 
$\omega_5(G^*)\ge\omega_5(G)$ and every non-trivial $3$-edge-cut in $G^*$ separates two uncolourable subgraphs of $G^*$.
Moreover, if $G^*$ has a $2$-factor with $k$ circuits of length $3$ and $m$ circuits of length $5$,
then also $G$ has a $2$-factor with at most $k$ circuits of length $3$ and at most $m$ circuits of length $5$.
\label{lemma3edgeCut2}
\end{lemma}
\begin{proof}
We can assume by the above lemmas that $G$ has no triangles and no $2$-edge-cut that separates colourable subgraph in $G$. 
We choose a non-trivial $3$-edge-cut that separates the smallest colourable subgraph. 
Let $v_1v_2$, $w_1w_2$, $x_1x_2$ be the cut-edges such that the vertices $v_1$, $w_1$, and $x_1$ are in the colourable subgraph. The vertices $v_1$, $v_2$, $w_1$, $w_2$, $x_1$, and $x_2$ are pairwise distincs otherwise there is a $2$-edge-cut in the graph.
Moreover, no two vertices of $v_1$, $w_1$, and $x_1$ are neighbours as this either contradicts the choice of
the $3$-edge-cut or the fact that $G$ has no triangles.

We create two new vertices $y_1$ and $y_2$, delete the edges $v_1v_2$, $w_1w_2$, and $x_1x_2$
and add the edges $v_1y_1$, $w_1y_1$, $x_1y_1$, $v_2y_2$, $w_2y_2$, and $x_2y_2$.
This creates two components: $G_1$ (containing $y_1$) and $G_2$ (containing $y_2$).

We set $G^*=G_2$. Similar arguments as in Lemma \ref{lemma2edgeCut2} can be used to prove the statement of the lemma.
\end{proof}

\section{Avoiding $5$-circuits}

We use a similar approach to bound $\omega_5(G)$ as we used to bound $\omega(G)$ in Section 2.

\begin{theorem}
Let $G$ be a $2$-edge-connected cubic graph $G$ on $n$ vertices not isomorphic to the Petersen graph. Graph $G$ has a $2$-factor that contains no triangle and at most $2(n-2)/15$ circuits of length $5$.
\label{thm2}
\end{theorem}
\begin{proof}
Suppose that $G$ is the smallest counterexample to the Theorem \ref{thm2}.
If $G$ is $3$-edge-colourable, then $G$ has an even $2$-factor. 
If $G$ is a snark, then
by Lemmas \ref{lemmaGirth2}, \ref{lemma2edgeCut2}, and  \ref{lemma3edgeCut2} the graph $G$
either reduces to the Petersen graph, or each $2$-edge-cut and non-trivial $3$-edge-cut separates uncolourable subgraphs and girth is $5$.
If $G$ reduces to the Petersen graph by any of these lemmas, then
it fulfils the Theorem \ref{thm2}.

The proof will follow the proof of Theorem \ref{thm1}. 
We modify the invariants and obtain the result by a similar 
argumentation. Where the argumentation is the same, we will refer to the proof of Theorem \ref{thm1}.

Let ${\cal P}'_1$ be the set of subgraphs of $G$ isomorphic to the graph $P_1$ and 
let ${\cal P}_3$ be the set of subgraphs of $G$ isomorphic to the graph $P_3$ that 
cannot be extended to a subgraph isomorphic to $P_1$
(both defined in Figure~\ref{fig}). Note that compared to the proof of Theorem \ref{thm1},
we do not care whether a subgraph isomorphic to $P_1$ can be extended to a subgraph isomorphic to $P_2$.
Let ${\cal P} = {\cal P}'_1 \cup {\cal P}_3$.
By Lemma \ref{ldis} we know that subgraphs in ${\cal P}$
are pairwise disjoint.

Let $M$ be a perfect matching of $G$ and let $F_M$ be the complementary $2$-factor.
Let ${\cal C}(M)$ be the set of circuits of $F_M$. We define the symbol $|C|_5$ as follows:
$|C|_5=1$ if $C$ is a circuit of length $5$ and $|C|_5=0$ otherwise.
Let 
$$
I(M)=\sum_{C\in {\cal C}(M)} |C|_5.
$$
Note that now the invariant $I(M)$ simply counts the number of circuits of length $5$. 
We will show that we can choose $M$ so that $I(M)$ is small enough. 

Let $C\in {\cal C(M)}$ and $S\in{\cal P}$. If $C$ is a $5$-circuit intersecting $S$, we define $I(C,S,M)=1$, 
otherwise $I(C,S,M)=0$.
(Note that due to properties of $P_1$ and $P_3$ a $5$-circuit cannot intersect two subgraphs from ${\cal P}$.)
Let $I(S,M)=\sum_{C\in{\cal C}(M)} I(C,S,M)$.
If $C\in {\cal C}(M)$ is a $5$-circuit that does not intersect any $S\in {\cal P}$, then we define
$I(C,M)=1$, otherwise $I(C,M)=0$.

We can rewrite the invariant $I$ as
\begin{eqnarray}
I(M)=\sum_{C\in{\cal C}(M)} I(C,M) + \sum_{S\in{\cal P}'_1} I(S,M) + \sum_{S\in{\cal P}_3} I(S,M). 
\label{eqn4}
\end{eqnarray}

Now we define a linear function to minimize over ${\cal M}(G)$. For each $S\in{\cal P}'_1$ let $e_S$ be one arbitrary edge on the boundary of $S$ (an edge with exactly one vertex belonging to $S$). 
Let ${\cal C}_5$ be the set of circuits of length $5$ in $G$ that do not intersect any subgraph from ${\cal P}$.
For each $C \in {\cal C}_5$ let $E_C$ be the set of edges on the boundary of $C$.

We find a point $p$ in ${\cal M}(G)$ such that
$$
f(p)=\left(\sum_{C \in {\cal C}_5} \frac{1}{4}\sum_{e\in E_C} p_e  \right) + \sum_{S\in {\cal P}'_1} p_{e_S}
$$
is minimal. The optimal point $p\in{\cal M}(G)$ is a perfect matching and satisfies
\begin{eqnarray}
f(p) \le 5/12 |{\cal C}_5|  + 1/3 |{\cal P}'_1|.
\label{eqn5}
\end{eqnarray}

Let $M$ be a perfect matching satisfying $f(M)=f(p)$. We bound
the summands in (\ref{eqn4}) one by one.

\medskip
\noindent {\bf Part 1:} First, we consider the value of $I(C,M)$ for $C\in{\cal C}(M)$. 
For a circuit $C\in {\cal C}(M)$ of length more than $5$ or for a circuit that intersects some subgraph from ${\cal P}$
we have $I(C,M) = 0$ by definition. Consider a circuit $C\in {\cal C}_5$ of length $5$. All boundary edges belong to $M$, that is $\sum_{e\in E_C} p_e=5$. By definition, we have $I(C,M)=1$. 
Because $\sum_{e\in E_C} p_e\ge 1$
we know that set $I(C,M)\le -1/4+1/4\sum_{e\in E_C} p_e$ and altogether
$$
\sum_{C\in{\cal C}(M)} I(C,M) \le 
\left(\sum_{C \in {\cal C}_5} \frac{1}{4}\sum_{e\in E_C} p_e \right)- \frac{1}{4}|{\cal C}_5|.
$$

\medskip
\noindent {\bf Part 2:} Consider a subgraph $S \in {\cal P}'_1$. Suppose that $p_{e_S}=1$. This means that both edges on the boundary of $S$ belong to $M$, and $I(S,M)=2$. 
On the other hand, if $p_{e_S}=0$, then $I(S,M)=1$ and
Together $I(S,M) \le p_{e_S} + 1$ and
$$
\sum_{S\in{\cal P}'_1} I(S,M)\le \left(\sum_{S\in {\cal P}'_1} p_{e_S}\right)+|{\cal P}'_1|.
$$

\medskip
\noindent {\bf Part 3:} Suppose we have a subgraph $S \in {\cal P}_{3}$. There is at most one $5$-circuit
intersecting $S$ (otherwise $S\in{P_1}$). Therefore $I(S,M) \leq 1$ and
$$
\sum_{S\in{\cal P}_3} I(S,M)\le |{\cal P}_3|.
$$

 Altogether $I(M)\le f(M) -|{\cal C}_5|/4 + |{\cal P}'_1| + |{\cal P}_3|$. 
Using (\ref{eqn5}), we have
$$
I(M) \le 1/6 |{\cal C}_5|  + 4/3 |{\cal P}'_1| +|{\cal P}_3|.
$$

Lemma \ref{l3} can be simplified to show that
for a $2$-edge-connected graph $G$,
$V(G) \ge 5/3 |{\cal C}_5| + 10 |{\cal P}'_1| + 9 |{\cal P}_{3}|.$
The function $I(M)$ counts the number of $5$-circuits in ${\cal C}(M)$. Therefore $\omega_5(G)\le I(M)$.
Altogether we have
\begin{eqnarray}
\frac{V(G)}{\omega_5(G)}\ge 
\frac
{5/3 |{\cal C}_5| + 10 |{\cal P}'_1| + 9 |{\cal P}_{3}|}
{1/6 |{\cal C}_5| + 4/3|{\cal P}'_1|+|{\cal P}_3|} \ge 7.5. 
\label{meq}
\end{eqnarray}

Note that the equality may hold only if all vertices of $G$ belong to some subgraph from ${\cal P}'_1$.
It is easily deducible that if all vertices are in a subgraph from ${\cal P}'_1$, then for such a graph $V(G)/\omega_5(G) = 10$. This improves our bound
to $(V(G)-2)/ \omega_5(G) \ge 7.5$, which is what we wanted to prove.
\end{proof}

We show that Theorem \ref{thm2} is the best possible. We can easily construct
a graph on $30k+2$ vertices that must contain $4k$ circuits of length $5$ in its $2$-factor.
We start with a cubic multigraph on $2$ vertices and replace the three edges
with three ``chains'' each consisting of $k$ subgraphs isomorphic to $P_1$. From formula (\ref{meq}) it can be deduced that these are all graphs for which Theorem \ref{thm2} is tight.
Note that these graphs satisfy Conjecture \ref{con}.
This gives additional support for this conjecture.

If we forbid subgraphs isomorphic to $P_1$ in $G$ after reductions using Lemmas
\ref{lemmaGirth2}, \ref{lemma2edgeCut2}, and \ref{lemma3edgeCut2} our bound improves 
to $9$, e.g. $3$-edge-connected graphs of girth $5$ without $3$-edge-cuts separating
colourable subgraphs. We believe that $9$ is correct  for $3$-edge-connected graphs even without these additional technical conditions.
\begin{conjecture}
Let $G$ be a $3$-edge-connected cubic graph $G$ on $n$ vertices not isomorphic to the Petersen graph. Graph $G$ has a $2$-factor that contains at most $n/9$ circuits of length $5$.
\end{conjecture}
\noindent Infinitely many graphs attain the bound in this conjecture: We take even number of copies of $P_3$
and connect them so that no $2$-edge-cut is created.

Similarly, if $G$ does not contain a subgraph isomorphic to $P_3$ nor such a subgraph is created 
by some reduction (note that $P_3$ is a subgraph of $P_1$), then our bound is improved to $n/10$.
\begin{theorem}
Let $G$ be a cyclically $4$-edge-connected cubic graph of girth $5$ on $n$ vertices not isomorphic to the Petersen graph. Then $G$ has a $2$-factor that contains at most $n/10$ circuits of length $5$.
\end{theorem}
\noindent 
An infinite family of cubic graphs is known which shows that the bound cannot be improved beyond $n/27$ \cite{5c}.


\begin{thebibliography}{99}
\bibitem{boyd}
S. Boyd, S. Iwata, K. Takazawa: 
Finding 2-Factors Closer to TSP Tours in Cubic Graphs,
SIAM J. Discrete Math. 27, 918--939, (2013).

\bibitem{devos} 
M. DeVos: Petersen graph conjecture - not too hard,
\url{http://garden.irmacs.sfu.ca/?q=op/petersen_graph_conjecture}

\bibitem{edmonds} J. Edmonds: Maximum matching and a polyhedron with (0, 1) vertices, 
J. Res. Nat. Bur. Standards Sect B. 69 B, 125--130, 1965.

\bibitem{expfactor}
L. Esperet, F. Kardo\v s, A. D. King, D. Kr\' al, S. Norine: 
Exponentially many perfect matchings in cubic graphs,
Advances in Mathematics 227, 1646--1664, 2011.

\bibitem{hag}
R. Häggkvist, S. McGuinness: Double covers of cubic graphs of oddness $4$, 
J. Combin. Theory Ser. B 93, 251--277, 2005.

\bibitem{huck}
A. Huck: Reducible configurations for the cycle double cover conjecture, 
Discrete Applied Mathematics 99 (1–3), 71--90, 2000.

\bibitem{huckkochol} A. Huck, M. Kochol: Five cycle double covers of some cubic graphs, 
J. Combinatorial Theory Ser. B 64 (1995) 119-125.

\bibitem{kochol}
M. Kochol: Smallest counterexample to the 5-flow conjecture has girth at least eleven,
Journal of Combinatorial Theory Series B 100,  381-389, 2010.

\bibitem{kral}
D. Král, E. Máčajová, J. Mazák, J.-S. Sereni: Circular edge-colorings of cubic graphs with girth six, J. Combin. Theory Ser. B 100, 351--358, (2010).

\bibitem{kundgen}
A. K\" undgen, R. B. Richter: On $2$-factors with long cycles in cubic graphs,
Ars Mathematica Contemporanea 4, 79--93, 2011.

\bibitem{jaeger} F. Jaeger: Nowhere-zero flow problems, 
in L. W. Beinecke, R. J. Wilson (ed.), Selected topics in graph theory 3, 71-95, Academic Press, San Diego, CA, 1988.

\bibitem{5c}
R. Lukoťka, E. Máčajová, J. Mazák, M. Škoviera: Circuits of length 5 in 2-factors of cubic graphs,
Discrete Mathematics 312, 2131--2134, 2012.

\bibitem{LMMS}
R. Lukoťka, E. Máčajová, J. Mazák, M. Škoviera: Small snarks with large oddness, 
{\tt arXiv:1212.3641 [cs.DM]}], 2012.

\bibitem{pgc}
V. V. Mkrtchyan, S. S. Petrosyan: Petersen graph conjecture:
\url{http://garden.irmacs.sfu.ca/?q=op/petersen_graph_conjecture}

\bibitem{petersen} 
J. Petersen: Die Theorie der regulären Graphen, 
Acta Math. 15, 193--200, 1891.

\bibitem{steffen} 
E. Steffen: Measurements of edge-uncolorability, 
Discrete Math. 280, 191--214, 2004.

\bibitem{snarks}
J. J. Watkins, R. J. Wilson: A survey snarks, 
in: Graph Theory, Combinatorics, and Applications Vol. 2, Y. Alavi et al. (eds.), Wiley, New York, 1129--1144, 1991.

\end{thebibliography}
\end{document}